\documentclass[12pt]{amsart}

\usepackage[english]{babel}
\usepackage[utf8]{inputenc}
\usepackage{amsmath}
\usepackage{amssymb}
\usepackage{amsfonts}
\usepackage{amsthm}
\usepackage{mathrsfs}
\usepackage[all]{xy}
\usepackage[pdftex]{graphicx}
\usepackage{color}
\usepackage{cite}
\usepackage{url}
\usepackage{indent first}
\usepackage[labelfont=bf,labelsep=period,justification=raggedright]{caption}
\usepackage[english]{babel}
\usepackage[utf8]{inputenc}
\usepackage{hyperref}
\usepackage[colorinlistoftodos]{todonotes}
\usepackage{tkz-fct}
\usepackage[margin=1in]{geometry} 
\usepackage{tikz}

\makeatletter
\def\amsbb{\use@mathgroup \M@U \symAMSb}
\makeatother

\usepackage{bbold}

\DeclareMathOperator{\D}{D}
\DeclareMathOperator{\Dioph}{Dioph}
\DeclareMathOperator{\Diop}{Dioph}
\DeclareMathOperator{\diop}{dioph}
\DeclareMathOperator{\dioph}{dioph}

\DeclareMathOperator{\Jac}{Jac}

\newcommand{\FF}{\amsbb{F}}

\newcommand{\NN}{\amsbb{N}}

\newcommand{\QQ}{\amsbb{Q}}

\newcommand{\ZZ}{\amsbb{Z}}

\newcommand{\mcO}{\mathcal{O}}

\theoremstyle{plain}
\newtheorem{thm}{Theorem}
\newtheorem{lemma}[thm]{Lemma}

\newtheorem{prop}[thm]{Proposition}

\theoremstyle{definition}
\newtheorem{defn}[thm]{Definition}

\theoremstyle{remark}
\newtheorem{rem}[thm]{Remark}

\numberwithin{equation}{section}
\numberwithin{thm}{section}

\begin{document}

\title{Diophantine Tuples over $\ZZ_p$}

\author{Nitya Mani}
\address{(Mani): Stanford University,  Department of Mathematics, Stanford, CA 94305}
\email{nityam@stanford.edu}
\author{Simon Rubinstein-Salzedo}
\address{(Rubinstein-Salzedo): Euler Circle, Palo Alto, CA 94306}
\email{simon@eulercircle.com}
\date{\today}

\maketitle

\begin{abstract}
For an element $r$ of a ring $R$, a $D(r)$ $m$-tuple is an $m$-tuple $(a_1,a_2,\ldots,a_m)$ of elements of $R$ such that for all $i,j$ with $i\neq j$, $a_ia_j+r$ is a perfect square in $R$. In this article, we compute and estimate the measures of the sets of $D(r)$ $m$-tuples in the ring $\ZZ_p$ of $p$-adic integers, as well as its residue field $\FF_p$.
\end{abstract}

\section{Introduction}
The study of Diophantine $m$-tuples has a long history, attracting the attention of the some of the most prominent mathematicians. The story begins in the 3rd century with Diophantus, who first observed that the quadruple $(\frac{1}{16},\frac{33}{16},\frac{17}{4},\frac{105}{16})$ has the curious property that the product of any two elements in the set was one less than a rational square. It was not until Fermat that the first quadruple of integers with this property was discovered, namely $(1,3,8,120)$: $1 \cdot 3 + 1 = 2^2$, $3 \cdot 8 + 1 = 5^2$, $\ldots$, $3 \cdot 120 + 1 = 19^2$. Euler found infinitely many Diophantine quadruples in the positive integers and also found a way of extending Fermat's set to a rational quintuple with this property, namely $(1,3,8,120,\frac{777480}{8288641})$. Such sets of numbers are now known as \textit{Diophantine $m$-tuples}, precisely defined below:

\begin{defn} For a commutative ring with unit $R$, an $m$-tuple $(a_1,\ldots,a_m) \in R^m$ is said to be a \textit{Diophantine $m$-tuple over $R$} if for every $i,j$ with $1\le i<j\le m$, $a_ia_j+1 \in \square(R)$, where $\square(R)$ denotes the set of square elements in ring $R$.
\end{defn}

Here, we will leverage some notation that enables us to include a more general family of sets than just the Diophantine $m$-tuples, as investigated in a number of works, such as~\cite{DUJ04}.

\begin{defn}
For $r \in R$, an $m$-tuple $(a_1, \ldots, a_m) \in R^m$ is said to be a \textit{$\D(r)$ $m$-tuple over $R$} if $$a_i a_j + r \in \square(R), \quad \text{for all } 1 \le i < j \le m.$$
We denote by $\Diop_m^r(R) \subset R^m$ the set of $\D(r)$ $m$-tuples over $R$. Note that $\Diop_m^1(R)$ is the set of Diophantine $m$-tuples over $R$. 
\end{defn}

In this notation, Diophantus's original example $(\frac{1}{16},\frac{33}{16},\frac{17}{4},\frac{105}{16})$ is in $\Dioph_4^1(\QQ)$, and after clearing denominators, we see that $(1,33,68,105)\in\Dioph_4^{256}(\ZZ)$. Fermat's first Diophantine (i.e.\ $\D(1)$) quadruple in the positive integers is $(1,3,8,120)\in\Dioph_4^1(\ZZ)$.

The problem of extending Fermat's triple $\{1,3,8\}$ to a quadruple was one of the early successes of Baker's results on linear forms in logarithms; see~\cite{BD69}. More recently, Gibbs in~\cite{Gibbs06} found a rational Diophantine sextuple, namely $(\frac{11}{192},\frac{35}{192},\frac{155}{27},\frac{512}{27},\frac{1235}{48},\frac{180873}{16})\in\Dioph_6^1(\QQ)$. Recently, He, Togb\'e, and Ziegler in~\cite{HTZ16} showed that there is no Diophantine quintuple in positive integers, i.e. that $\Dioph_5^1(\ZZ) = \varnothing$.

A curious property about Diophantine tuples is that any triple, over any ring (or even any semiring) can be extended to a Diophantine quadruple. This is special to the case of $\D(1)$ (or more generally $\D(k^2)$), as $\D(r)$ triples do not generally extend to $\D(r)$ quadruples in interesting ways. (Of course, one can always extend any $\D(r)$ $m$-tuple to a $\D(r)$ $(m+1)$-tuple by appending 0, but this is rather trivial.) An open question of considerable interest is whether or not every Diophantine triple $(a,b,c)$ in the positive integers can be uniquely extended to a Diophantine quadruple $(a,b,c,d)$ with $d\in\ZZ$ and $d>\max(a,b,c)$.

Most of the past work in Diophantine and $\D(r)$ tuples has been in the case of $R=\ZZ$ (or even $\NN$) and $R=\QQ$. There have been a few papers studying the case where $R$ is the ring of integers in a quadratic field (e.g.~\cite{Dujella97,Franusic04}) and a cubic field (e.g.~\cite{Franusic13}). The question of Diophantine tuples over finite fields was considered in~\cite{DK16}. 
In recent years, there has been a flurry of activity on Diophantine and $\D(r)$ tuples. Instead of summarizing more of it, we instead refer the interested reader to Dujella's extensive bibliography of the field, which can be found at~\url{https://web.math.pmf.unizg.hr/~duje/ref.html}.

In the present paper, we investigate some questions similar in spirit to those of~\cite{DK16}. In~\cite{DK16} and frequently in this literature, it is assumed that elements of a Diophantine tuple must be distinct. However, our main goal is an understanding of the Haar measure of tuples over $\ZZ_p$. To do so, it is often convenient to reduce modulo $p$, so we do not require our tuples to have distinct elements.

The main goal of this article is to evaluate or approximate the Haar measure of the set $\Diop_m^r(\ZZ_p)$, for certain values of $m$, $r$, and $p$. Then our goal is to evaluate $\diop_m^r(\ZZ_p)$ defined below:

\begin{defn}
For $r \in R$, let $\diop_m^r(R) = \mu(\Diop_m^r(R))$ be the measure of the set of $\D(r)$ $m$-tuples in $R^m$, where $\mu$ is the product measure induced by a choice of finite measure $\mu$ on $R$. 
\end{defn}

Throughout this article, we will assume that $R$ is a commutative ring equipped with a probability measure $\mu$, i.e.\ one in which $\mu(R)=1$. The main case we will be focusing on is the case in which $R$ is the ring $\ZZ_p$ of $p$-adic integers equipped with its Haar measure $\mu$, normalized so that $\mu(\ZZ_p)=1$. We will also look at the case where $R=\mcO_K$ is the valuation ring of a finite extension $K$ of $\QQ_p$, again equipped with its normalized Haar measure. We will also consider the case where $R=\FF_q$ is a finite field, equipped with the normalized counting measure, where the normalization is such that $\mu(\FF_q)=1$.

Unless otherwise specified, we take $a, b, c, d \in \ZZ_p$. All of our results over $\ZZ_p$ can be easily modified to obtain bounds over $\FF_p$ with respect to the counting measure on $\FF_p$ and vice versa, modulo lower order terms.

In the case where $m$ is small, we can usually give an exact value for $\diop_m^r(\ZZ_p)$, whereas when $m\ge 4$, we can only give an asymptotic as $p\to\infty$, with a lower-order error term. One of our main techniques is to study the reduction to $\diop_m^r(\FF_p)$ and then use results about character sums in $\FF_p$. More precisely, we note that $\diop_m^r(\ZZ_p)$ is quite close to \[\frac{1}{2^{\binom{m}{2}}}\int_{\ZZ_p^m} \prod_{1\le i<j\le m}\left(1+\left(\frac{a_ia_j+1}{p}\right)\right)\,d\mu,\] which can be calculated exactly for $m\le 3$ and approximated up to an error term of the form $O(p^{-1/2})$ as $p\to\infty$ when $m\ge 4$. Here and elsewhere, $\left(\frac{\cdot}{p}\right)$ denotes the Legendre symbol on $\ZZ_p$.

We proceed as follows in this article. In Section~\ref{s:z2} we collect some notation and observe that the Haar measure of Diophantine pairs over $\ZZ_2$ is $\frac13$. In Section~\ref{s:pairs}, we prove Theorem~\ref{t:oddpairs} and exactly compute the Haar measure of Diophantine $D(r)$ pairs over $\ZZ_p$ for all odd primes $p$. We observe in Section~\ref{sec:Z3} that when $p = 3$, we can compute the Haar measure exactly of $D(r)$ $m$-tuples for all $\left( \frac{r}{3} \right) = 1$ and $m \ge 2$. We proceed in Section~\ref{sec:D3r} to give asymptotic bounds in Theorem~\ref{t:3asympind} on $\diop_3^r(\ZZ_p)$, exactly computing $\diop_3^r(\FF_p)$ in Theorem~\ref{t:3fp} as an intermediate step. We extend our asymptotic analysis of $\Diop_m^r(\ZZ_p)$ to $m \ge 4$ in Section~\ref{sec:m4}. Finally, we are able to carry over many of our methods when we replace $\ZZ_p$ by a finite extension thereof as we observe in Section~\ref{s:extension}, for example exactly computing the Haar measure of Diophantine $D(r)$ pairs in Theorem~\ref{t:oddpairs2} when the residue characteristic is odd.

\section{$\Diop_2^1(\ZZ_2)$}\label{s:z2}

Before stating any results, let us recall some notation that will be used throughout the article.

\begin{defn}
For $a \in \ZZ_p$, we let $v_p(a)$ denote the \textit{$p$-adic valuation} of $a$, given by 
$$v_p(a) = \sup\left\{ k \in \ZZ_{\ge 0} : a \in p^k\ZZ_p \right\}.$$
\end{defn}

\begin{defn}
More generally, if $R = \mcO_K$ is a \textit{valuation ring}, we let $v_R$ denote the normalized valuation so that if $\pi \in \mcO_K$ is a uniformizer, then $v_R(\pi) = 1$.
\end{defn}

Note that $(a, b) \in \ZZ_p^2$ is a Diophantine pair over $\ZZ_2$ if $ab + 1 \in \square(\ZZ_2)$. An element $\alpha = 2^{v_2(\alpha)} \beta \in \square(\ZZ_2)$ if 
$$v_2(\alpha) \equiv 0 \pmod 2\qquad \text{and} \qquad \beta \equiv 1 \pmod 8.$$
\begin{prop}
The Haar measure of Diophantine pairs over $\ZZ_2$ is given by
$$\diop_2^1(\ZZ_2) = \frac13.$$
\end{prop}
\begin{proof}
We consider cases based on the residues of $ab$ modulo $8$ and the $2$-adic valuation of $ab + 1$ in $\ZZ_2$. To that end, we will define the measures of disjoint subsets of $\ZZ_2^2$. For $k \ge 0$, let 
$$A_{2k} = \{ (a, b) \in \Diop_2^1(\ZZ_2) : v_2(ab + 1) = 2k\}.$$
If $ab + 1 \not \equiv 0 \pmod 4$, $ab + 1 = \square \in \ZZ_2$ exactly when $ab \equiv 0 \pmod 8$. This occurs when either one of $a, b$ is $0 \pmod 8$ or $$(a, b) \pmod 8 \in \{(\pm 2, 4), (4, \pm 2), (4, 4)\}.$$ This yields
$$A_0 = \frac{7 + 7 + 1 + 2 + 2 + 1}{64} = \frac{5}{16}.$$
If $v_2(ab + 1) = 2$, then $ab + 1 = \square \in \ZZ_2$ exactly when $ab + 1 \equiv 4 \pmod {32}$. Considering residue classes of $a, b \pmod {32}$ yields $A_2 = \frac12 \cdot \frac{1}{32}$. Similarly we obtain for $v_2(ab + 1) = 2k$ for $k > 1$, that $ab + 1 = \square$  exactly when 
$$ab + 1 = 2^{2k} \cdot \alpha, \quad \alpha \equiv 1 \pmod 8$$
or in other words when
$$ab + 1 \equiv 2^{2k} \mod 2^{2k+ 3}.$$
The above result implies that for $k \ge 1$, we have that
$$A_{2k} = \frac{1}{2} \cdot \frac{1}{2^{2k+3}} = \frac{1}{2^{2k+4}}.$$
We therefore obtain the desired result:
$$\diop_2^1(\ZZ_2) = \mu \{(a, b) \in \ZZ_2^2 : ab + 1 \in \square(\ZZ_2)\} = A_0 + \sum_{k = 1}^{\infty} A_{2k} = \frac{5}{16} + \sum_{k = 1}^{\infty} \frac{1}{2^{2k+4}} = \frac{5}{16} + \frac{1}{48} = \frac{1}{3}.$$
\end{proof}

\section{$\Diop_2^r(\ZZ_p)$}\label{s:pairs}

\begin{thm}\label{t:oddpairs}
Let $p$ be an odd prime and let $\alpha = v_p(r)$ with $r = p^{\alpha} s$. Then, we have the following:
\footnotesize
$$\diop_2^r(\ZZ_p) = 
\begin{cases}
\frac12 + \frac{1}{p(p+1)} & \left( \frac{r}{p} \right) = 1, \\
\frac12 - \frac{1}{p+1} & \left( \frac{r}{p} \right) = -1, \\
\frac12 - \frac{p-1}{2(p+1)^2} -\frac{(\alpha + 2)
}{2 (p+1)^2 p^{\alpha - 1}} -\frac{1}{2 (p+1)^2 p^{\alpha}} + \frac{(\alpha - 1)  }{2 (p+1)^2 p^{\alpha + 1}}  & \alpha \equiv 1 \pmod 2, \\
\frac{1}{2} + \frac{(\alpha + 1)(p-1)^2}{2p^{\alpha + 2}} - \frac{2p - 1}{2(p+1)^2} +\frac{1}{2 (p+1)^2 p}- \frac{ (\alpha + 1)}{2 (p+1)^2p^{\alpha - 2}} \\
\quad +\frac{(\alpha -1) }{2 (p+1)^2 p^{\alpha}} -\frac{1}{2 (p+1)^2 p^{\alpha + 1}}  -\frac{1}{2 (p+1)^2 p^{\alpha + 2}}
& \alpha > 0,\, \alpha \equiv 0 \pmod 2, \left( \frac{s}{p} \right) = 1, \\
\frac{1}{2} + \frac{\alpha (p  - 1)^2 + p^2 + 1}{2p^{\alpha + 2}} - \frac{2p - 1}{2(p+1)^2} +\frac{1}{2 (p+1)^2 p}- \frac{ (\alpha + 1)}{2 (p+1)^2p^{\alpha - 2}} \\
\quad +\frac{(\alpha -1) }{2 (p+1)^2 p^{\alpha}} -\frac{1}{2 (p+1)^2 p^{\alpha + 1}}  -\frac{1}{2 (p+1)^2 p^{\alpha + 2}}
& \alpha > 0,\, \alpha \equiv 0 \pmod 2, \left( \frac{s}{p} \right) = -1,
\end{cases}
$$
\normalsize
where $(\frac{a}{b})$ denotes the Legendre symbol.
In particular, we have that the measure of the Diophantine pairs over $\ZZ_p$ is given by
$$\diop_2^1(\ZZ_p) = \frac12 + \frac{1}{p(p+1)}.$$
\end{thm}

\begin{rem}
We will repeatedly use the following characterization of $\square(\ZZ_p)$. For an odd prime $p$, $\alpha \in \square(\ZZ_p)$ exactly if $\alpha = p^{2k}\beta$ for some nonnegative integer $k$ and some $\beta \in \ZZ_p$ with $\left(\frac{\beta}{p} \right) = 1$. 
\end{rem}

\begin{lemma}\label{l:pairnonresidue}
Let $p$ be an odd prime and take $r \in \ZZ_p$ so that $\left(\frac{r}{p} \right) = -1$. Then, 
$$\dioph_2^r(\ZZ_p) = \frac12 - \frac{1}{p+1}.$$
\end{lemma}
\begin{proof}
We compute $\dioph_2^r(\ZZ_p)$ by summing the measures of disjoint subsets of $\Diop_2^r(\ZZ_p)$. For $k \ge 0$, let 
$$A_{k} = \left\{ (a, b) \in \Diop_2^r(\ZZ_p) : v_p(ab + r) = 2k \right\}.$$
First, we consider $A_0$. If $(a, b) \in A_0$, then $ab + r = \beta$ where $\left(\frac{\beta}{p} \right) = 1$. Thus 
$$ab = \beta - r \not \equiv 0 \pmod p,$$
since $\left(\frac{r}{p} \right) = -1$. This implies that for each choice of nonzero residue class of $a \pmod p$ and every nonzero square residue class choice of $\beta \pmod p$, there is exactly one choice of residue class $\pmod p$ for $b$ that results in $ab + r \in A_0$. Since there are $(p-1)/2$ nonzero square residue classes $\pmod p$, we have that 
$$\mu(A_0) = \frac{1}{p^2} \cdot (p-1) \cdot \frac{(p-1)}{2} = \frac{(p-1)^2}{2p^2}.$$
Similarly, for $k \ge 1$, if $(a, b) \in A_{k}$, then $ab + r = p^{2k} \cdot \beta$ where $\left(\frac{\beta}{p} \right) = 1$. Here 
$$ab = p^{2k} \beta - r \not \equiv 0 \pmod p$$
since $r \not \equiv 0 \pmod p$. 
such that $a \not \equiv 0 \pmod p$ and the resulting possibilities for $b \pmod{p^{2k+1}}$, we obtain
$$\mu(A_{2k}) = \frac{1}{p^{4k + 2}} \cdot (p^{2k+1} - p^{2k}) \cdot \frac{p-1}{2} = \frac{(p-1)^2}{2p^{2k + 2}}.$$
This yields the desired result:
$$\diop_2^r(\ZZ_p) = \sum_{k = 0}^{\infty} \mu(A_k) = \sum_{k = 0}^{\infty} \frac{(p-1)^2}{2p^{2k + 2}} = \frac{(p-1)^2}{2p^2} \cdot \frac{p^2}{p^2 - 1} = \frac{(p-1)}{2(p+1)} = \frac12 - \frac{1}{p+1}. $$
\end{proof}

\begin{lemma}\label{l:pairresidue}
Let $p$ be an odd prime and take $r \in \ZZ_p$ so that $\left(\frac{r}{p} \right) = 1$. Then, 
$$\dioph_2^r(\ZZ_p) = \frac12 + \frac{1}{p(p+1)}.$$
\end{lemma}
\begin{proof}
As above, we compute $\dioph_2^r(\ZZ_p)$ by summing the measures of disjoint subsets of $\Diop_2^r(\ZZ_p)$. For $k \ge 0$, let 
$$A_{k} = \left\{ (a, b) \in \Diop_2^r(\ZZ_p) : v_p(ab + r) = 2k \right\}.$$
We first consider $A_0$. If $(a, b) \in A_0$, then $ab + r = \beta$ where $\left(\frac{\beta}{p} \right) = 1$. Since $\left(\frac{r}{p} \right) = 1$, if either $a, b \equiv 0 \pmod p$, then $ab + r \equiv r \pmod p$ is nonzero square in $\FF_p$ and thus a square in $\ZZ_p$. We note that 
$$\mu ( \{ (a, b) \in \ZZ_p^2 : ab \equiv 0 \pmod p \}) = \frac{2p - 1}{p^2}.$$
Else, we write 
$ab = \beta - r \pmod p$
This expression is nonzero as long as $\beta \neq r$. This implies that for each choice of nonzero residue class of $a \pmod p$ and every nonzero square residue class choice of $\beta \pmod p$ except $r \pmod p$, there is exactly one choice of residue class $\pmod p$ for $b$ that results in $ab + r \in A_0$. Since there are $(p-1)/2$ nonzero square residue classes $\pmod p$, we have that 
$$\mu(A_0) =  \frac{2p - 1}{p^2} + \frac{1}{p^2} \cdot (p-1) \cdot \left( \frac{(p-1)}{2} - 1 \right) =  \frac{p^2 + 1}{2p^2}.$$
Similar to Lemma~\ref{l:pairnonresidue}, for $k \ge 1$, if $(a, b) \in A_{k}$, then $ab + r = p^{2k} \cdot \beta$ where $\left(\frac{\beta}{p} \right) = 1$. Here 
$$ab = p^{2k} \beta - r \not \equiv 0 \pmod {p^{2k + 1}}.$$
since $r \not \equiv 0 \pmod p$. Considering all residue class of $a$ $\pmod p^3$ such that $a \not \equiv 0 \pmod p$ and the resulting possibilities for $b \pmod p^3$, we obtain
$$\mu(A_{2k}) = \frac{1}{p^{4k + 2}} \cdot (p^{2k+1} - p^{2k}) \cdot \frac{p-1}{2} = \frac{(p-1)^2}{2p^{2k + 2}}.$$
This yields the desired result:
\begin{align*}
\diop_2^r(\ZZ_p) &= \sum_{k = 0}^{\infty} \mu(A_k) \\
&= \frac{p^2 + 1}{2p^2} + \sum_{k = 1}^{\infty} \frac{(p-1)^2}{2p^{2k + 2}} \\
&=  \frac{p^2 + 1}{2p^2} + \frac{p-1}{2p^2(p+1)} \\
&= \frac{p^2 + p + 2}{2p(p+1)} \\
&= \frac{1}{2} + \frac{1}{p(p+1)}.
\end{align*}
\end{proof}

\begin{lemma}\label{l:betagalpha}
Fix some $r \in \ZZ_p$ such that $p\mid r$, with $v_p(r) = \alpha>0$. Then for all $\beta \equiv 0 \pmod 2$ such that $\beta > \alpha$, if 
$$B_{\beta} = \{(a, b) \in \Diop_2^r(\ZZ_p) : v_p(ab + r) = \beta \},$$
then $$\mu(B_{\beta}) = \frac{(p-1)(p^{\alpha +1} - 1)}{2p^{\beta + 2}}.$$
\end{lemma}
\begin{proof}
We can write $r = p^{\alpha} s$ with $s\in\ZZ_p^\times$, and note that $\alpha \ge 1$. Note that for any $(a, b) \in B_{\beta}$, $ab + r = p^{\beta} q^2$ for some $q \in \ZZ_p^\times$. Thus, $(a, b) \in B_{\beta}$ iff there exists some $q \in \ZZ_p^\times$ such that 
$$ab \equiv p^{\alpha} \left( p^{\beta - \alpha} q^2 - s \right) \pmod{p^{\beta + 1}}.$$
Thus, we can compute $\mu(B_{\beta})$ using a probability over the residue classes of $a, b$ in $\ZZ/p^{\beta + 1}\ZZ$.
We can write $a = p^{e_a} a', b = p^{e_b} b'$ where $a', b' \in \ZZ_p^\times$, and $e_a + e_b = \alpha$. Then, there are $p^{\beta + 1 - e_a} - p^{\beta - e_a}$ choices for the residue $a \pmod p^{\beta + 1}$, such that there exists some $q$ and choice(s) of $b \pmod p^{\beta + 1}$ to satisfy the above equivalence.

For every such choice of $a$, there are $\frac{p-1}{2}$ choices for $b' \pmod{p^{\beta - \alpha + 1}}$, corresponding to the $\frac{p-1}{2}$ choices for the residue $q^2 \pmod p$. There are $p^{\alpha}$ ways to lift this residue $\pmod{p^{\beta + 1 - \alpha}}$ to a residue $\pmod{p^{\beta + 1}}$, yielding $p^{\alpha} \cdot (p-1)/2$ residue choices of $b \pmod{p^{\beta + 1}}$ for each choice of $a$, such that there exists some $q \in \ZZ_p$ where the above equivalence holds. Since $e_a$ can be any integer from $0$ to $\alpha$, inclusive, this gives the desired result:
\begin{align*}
\mu(B_{\beta}) &= \sum_{e_a = 0}^{\alpha} \frac{p^{\beta - e_a}(p-1) p^{\alpha} (p-1)}{2 p^{2(\beta + 1)}} \\
&= \sum_{e_a = 0}^{\alpha} \frac{p^{\alpha- e_a}(p-1)^2}{2 p^{\beta + 2}} \\
&= \frac{(p-1)^2}{2p^{\beta - \alpha + 2}} \sum_{e_a = 0}^{\alpha} \frac{1}{p^{e_a}} \\
&= \frac{(p-1)^2}{2p^{\beta - \alpha + 2}} \cdot \frac{p^{\alpha +1} - 1}{p^{\alpha}(p-1)} \\
&= \frac{(p-1)(p^{\alpha +1} - 1)}{2p^{\beta + 2}}.
\end{align*}
\end{proof}

\begin{lemma}\label{l:betalalpha}
Fix some $r \in \ZZ_p$ such that $p\mid r$, with $v_p(r) = \alpha$. Let $\beta \equiv 0 \pmod 2$ be such that $\beta < \alpha$, and set 
$$B_{\beta} = \{(a, b) \in \Diop_2^r(\ZZ_p) : v_p(ab + r) = \beta \}.$$
Then $$\mu(B_{\beta}) = \frac{(\beta + 1)(p-1)^2}{2 p^{\beta + 2}}.$$
\end{lemma}
\begin{proof}
As above, we write $r = p^{\alpha} s$ with $s\in\ZZ_p^\times$, and note that $\alpha \ge 1$. Note that for any $(a, b) \in B_{\beta}$, $ab + r = p^{\beta} q^2$ for some $q \in \ZZ_p^\times$. Thus, $(a, b) \in B_{\beta}$ iff there exists some $q \in \ZZ_p^\times$ such that 
$$ab \equiv p^{\beta} \left( q^2 - p^{\alpha - \beta}s \right) \pmod{p^{\alpha + 1}}.$$
Thus, we can compute $\mu(B_{\beta})$ using a probability over the residue classes of $a, b$ in $\ZZ/p^{\beta + 1}\ZZ$. We can write $a = p^{e_a} a', b = p^{e_b} b'$ where $a', b' \in \ZZ_p^\times$, and $e_a + e_b = \beta$. Then, there are $p^{\alpha + 1 - e_a} - p^{\alpha - e_a}$ choices for the residue $a \pmod {p^{\alpha + 1}}$, such that there exists some $q$ and choice(s) of $b \pmod {p^{\alpha + 1}}$ to satisfy the above equivalence.

For every such choice of $a$, there are $\frac{p-1}{2}$ choices for $b' \pmod{p}$, corresponding to the $\frac{p-1}{2}$ choices for the residue $q^2 \pmod p$. There are $p^{\alpha - e_b}$ ways to lift this residue $b' \pmod{p}$ to a residue of $b \pmod{p^{\alpha + 1}}$, yielding $p^{\alpha - e_b} \cdot (p-1)/2$ residue choices of $b \pmod{p^{\beta + 1}}$ for each choice of $a$, such that there exists some $q \in \ZZ_p$ where the above equivalence holds. Since $e_a$ can be any integer from $0$ to $\alpha$, inclusive, this gives the desired result:
\begin{align*}
\mu(B_{\beta}) &= \sum_{e_a = 0}^{\beta} \frac{p^{\alpha - e_a}(p-1) p^{\alpha - e_b}(p-1)/2}{p^{2(\alpha + 1)}} \\
&= \sum_{e_a = 0}^{\beta} \frac{p^{2\alpha - e_a - e_b}(p-1)^2}{2 p^{2(\alpha + 1)}} \\
&= \sum_{e_a = 0}^{\beta} \frac{p^{-\beta}(p-1)^2}{2 p^{2}} \\
&= \frac{(\beta + 1)(p-1)^2}{2 p^{\beta + 2}}.
\end{align*}
\end{proof}

\begin{lemma}\label{l:betaealpha}
Fix some $r \in \ZZ_p$ with $v_p(r) = \alpha > 0$ and let $r = p^{\alpha} s$. If
$$B_{\alpha} = \{(a, b) \in \Diop_2^r(\ZZ_p) : v_p(ab + r) = \alpha \},$$
then $$\mu(B_{\alpha}) = 
\begin{cases}
\frac{(\alpha + 1)(p-1)^2}{2p^{\alpha + 2}} & \alpha\equiv 0\pmod{2},\ \left( \frac{s}{p} \right) = 1, \\
\frac{\alpha (p  - 1)^2 + p^2 + 1}{2p^{\alpha + 2}} & \alpha\equiv 0\pmod{2},\ \left( \frac{s}{p} \right) = -1, \\ 0 & \alpha\equiv 1\pmod{2}.
\end{cases}
$$
\end{lemma}

\begin{proof}
Note that if $\alpha\equiv 1\pmod{2}$, then no element $x\in\ZZ_p$ with $v_p(x)=\alpha$ is a square. So now assume that $\alpha \equiv 0 \pmod 2$. For any $(a, b) \in B_{\alpha}$, we have $ab + r = p^{\alpha} q^2$ for some $q \in \ZZ_p^\times$. Rearranging, we find that any $(a, b) \in B_{\alpha}$ satisfies $ab = p^{\alpha} (q^2 - s)$ for some $q\in\ZZ_p^\times$. Note that if $v_p(q^2 - s) = \gamma$, we can check if $(a, b) \in \Diop_2^r(\ZZ_p)$ by seeing if $ab + r$ is a square $\pmod{p^{\alpha + \gamma + 1}}$.
Thus, we consider two cases for $\gamma$.

If $\gamma = 0$, then $q^2 - s\in\ZZ_p^\times$, and we can check if $ab + r \in \Diop_2^r(\ZZ_p)$ by computing the residue $\pmod {p^{\alpha + 1}}$.
As above, we write $a = p^{e_a} a', b = p^{e_b} b'$ where $a', b' \in \ZZ_p^\times$ and $e_a + e_b = \alpha$.
There are $p^{\alpha - e_a}(p-1)$ choices for $a \pmod {p^{\alpha + 1}}$ such that there exists $b \in \ZZ_p$ with $ab + r \in \Diop_2^r(\ZZ_p)$. For each such fixed choice of a residue class for $a$, there are $\frac{p-1}{2}$ choices for $b' \pmod p$ if $\left( \frac{s}{p} \right) = -1$, and $\frac{p-3}{2}$ choices for $b' \pmod p$ if $\left( \frac{s}{p} \right) = 1$, and $p^{\alpha - e_b}$ ways to lift each such choice of $b' \pmod p$ to a residue $b \pmod {p^{\alpha + 1}}$ such that 
$$ab \equiv p^{\alpha}(q^2 - s) \pmod{p^{\alpha+1}}$$
for some $q\in\ZZ_p^\times$. In particular, this implies that if $\left( \frac{s}{p} \right) = -1$, then
$$\mu(B_{\alpha}) = \sum_{e_a = 0}^{\alpha} \frac{(p-1)^2}{2p^{\alpha + 2}} = \frac{(\alpha + 1)(p-1)^2}{2p^{\alpha + 2}}.$$ 
If $\gamma \ge 1$, then $\left( \frac{s}{p} \right) = 1$. We can handle all of these cases together, since the residue of $q^2 - s$ is restricted $\pmod p$, but not modulo higher powers of $p$.
Let $q^2 - s = p^{\gamma} z$ where $z\in\ZZ_p^\times$. Evaluating whether $ab + r = p^{\alpha} q^2 \in \Diop_2^r(\ZZ_p)$ in these cases reduces to a check of $ab \pmod {p^{\alpha + \gamma + 1}}$. Using the same decomposition of $a, b$ as above,  $a = p^{e_a} a', b = p^{e_b} b'$ with $e_a + e_b = \alpha + \gamma$, we have $p^{\alpha + \gamma - e_a}(p-1)$ choices for $a \pmod {p^{\alpha + \gamma + 1}}$ such that there exists some $b$ making the following equivalence hold:
$$ab \equiv p^{\alpha + \gamma} z \pmod{p^{\alpha + \gamma + 1}}, \quad z\in\ZZ_p^\times.$$
For each such choice of $a$, there are $p^{\alpha + \gamma - e_b}(p - 1)$ choices for the residue $b \pmod{p^{\alpha + \gamma + 1}}$ such that there is some $z$ satisfying the above equivalence. This yields a
$$\frac{p^{\alpha + \gamma - e_a}(p-1)p^{\alpha + \gamma - e_b}(p-1)}{p^{2(\alpha + \gamma + 1)}} = \frac{(p-1)^2}{p^{\alpha + \gamma + 2}}$$
fraction of the residue classes $(a, b) \pmod p^{\alpha + \gamma + 1}$ such that $ab + r \in \Diop_2^r(\ZZ_p)$ with $ab + r = p^{\alpha} q^2$ such that $v_p(q^2 - s) = \gamma$.
Combining these disjoint cases gives the desired equality when $\left( \frac{s}{p} \right) = 1$:
\begin{align*}
\mu(B_{\alpha}) &= \sum_{e_a = 0}^{\alpha} \frac{(p-3)(p-1)}{2p^{\alpha + 2}} + \sum_{\gamma = 1}^{\infty} \sum_{e_a = 0}^{\alpha + \gamma} \frac{(p-1)^2}{p^{\alpha + \gamma + 2}} \\
&= \frac{(\alpha + 1)(p-3)(p-1)}{2p^{\alpha + 2}} + \sum_{\gamma = 1}^{\infty} \frac{(\alpha + \gamma + 1)(p-1)^2}{p^{\alpha + \gamma + 2}} \\
&= \frac{(\alpha + 1)(p-3)(p-1)}{2p^{\alpha + 2}} + \frac{(p-1)^2}{p} \sum_{x = \alpha + 2}^{\infty} \frac{x}{p^x}  \\
&= \frac{(\alpha + 1)(p-3)(p-1)}{2p^{\alpha + 2}} + \frac{(p-1)^2}{p} \cdot \frac{1 + (\alpha + 2)(p-1)}{(p-1)^2 p^{\alpha + 1}} \\
&= \frac{(\alpha + 1)(p-3)(p-1)}{2p^{\alpha + 2}} + \frac{1 + (\alpha + 2)(p-1)}{p^{\alpha + 2}} \\
&= \frac{(\alpha + 1)(p-3)(p-1) + 2 + 2(\alpha + 2)(p-1)}{2p^{\alpha + 2}} \\
&= \frac{\alpha (p  - 1)^2 + p^2 + 1}{2p^{\alpha + 2}}.
\end{align*}
\end{proof}

\begin{proof}[Proof of Theorem~\ref{t:oddpairs}] 
The result follows immediately by applying Lemmas~\ref{l:pairnonresidue} and~\ref{l:pairresidue} when $\left( \frac{r}{p} \right) \neq 0$. 
Next, let $v_p(r) = \alpha > 0$ and let $r = p^{\alpha} s$. We take two cases based on the parity of $\alpha$ and leverage the results of Lemmas~\ref{l:betagalpha}, ~\ref{l:betalalpha}, and~\ref{l:betaealpha}. First suppose $\alpha$ is odd. Then, 
\begin{align*}
\diop_2^r(\ZZ_p) &= \sum_{\substack{0\le \beta\le\alpha-1 \\ \beta \equiv 0 \pmod 2}} \mu(B_{\beta}) + \sum_{\substack{\beta\ge \alpha + 1 \\ \beta \equiv 0 \pmod 2}} \mu(B_{\beta})\\
&= \sum_{\substack{0\le \beta\le\alpha-1 \\ \beta \equiv 0 \pmod 2}} \frac{(\beta + 1)(p-1)^2}{2 p^{\beta + 2}} + \sum_{\substack{\beta\ge \alpha + 1 \\ \beta \equiv 0 \pmod 2}}  \frac{(p-1)(p^{\alpha +1} - 1)}{2p^{\beta + 2}} \\
&= \frac{(p-1)^2}{2p} \sum_{x = 0}^{(\alpha - 1)/2} \frac{(2x + 1)}{p^{2x + 1}} + \frac{(p-1)(p^{\alpha + 1} - 1)}{2p^2} \sum_{x = (\alpha + 1)/2}^{\infty}  \frac{1}{p^{2x}} \\
&= \frac12 - \frac{p-1}{2(p+1)^2} -\frac{(\alpha + 2)
}{2 (p+1)^2 p^{\alpha - 1}} -\frac{1}{2 (p+1)^2 p^{\alpha}} + \frac{(\alpha - 1)  }{2 (p+1)^2 p^{\alpha + 1}}.
\end{align*}
Next suppose that $\alpha$ is even. Then 
\begin{align*}
\diop_2^r(\ZZ_p) &= \sum_{\substack{0\le \beta\le\alpha-2 \\ \beta \equiv 0 \pmod 2}} \mu(B_{\beta})+ \mu(B_{\alpha}) + \sum_{\substack{\beta\ge \alpha + 2 \\ \beta \equiv 0 \pmod 2}} \mu(B_{\beta}) \\
&= \sum_{\substack{0\le \beta\le\alpha-2 \\ \beta \equiv 0 \pmod 2}}  \frac{(\beta + 1)(p-1)^2}{2 p^{\beta + 2}}  + \mu(B_{\alpha}) + \sum_{\substack{\beta\ge \alpha + 2 \\ \beta \equiv 0 \pmod 2}} \frac{(p-1)(p^{\alpha +1} - 1)}{2p^{\beta + 2}}  \\
&= \mu(B_{\alpha}) + \frac{(p-1)^2}{2p} \sum_{x = 0}^{\alpha/2 - 1} \frac{(2x + 1)}{p^{2x + 1}} + \frac{(p-1)(p^{\alpha + 1} - 1)}{2p^2} \sum_{x = \alpha/2 + 1}^{\infty}  \frac{1}{p^{2x}}\\
&= \mu(B_{\alpha}) + \frac12 - \frac{2p - 1}{2(p+1)^2} +\frac{1}{2 (p+1)^2 p}
- \frac{ (\alpha + 1)}{2 (p+1)^2p^{\alpha - 2}} +\frac{(\alpha -1) }{2 (p+1)^2 p^{\alpha}} \\
&\qquad \qquad -\frac{1}{2 (p+1)^2 p^{\alpha + 1}}  -\frac{1}{2 (p+1)^2 p^{\alpha + 2}}.
\end{align*}
Thus, when $\left( \frac{s}{p} \right) = 1$, 
$$\diop_2^r(\ZZ_p) = \frac{1}{2} + \frac{(\alpha + 1)(p-1)^2}{2p^{\alpha + 2}} - \frac{2p - 1}{2(p+1)^2} +\frac{1}{2 (p+1)^2 p}- \frac{ (\alpha + 1)}{2 (p+1)^2p^{\alpha - 2}} 
$$
$$\quad +\frac{(\alpha -1) }{2 (p+1)^2 p^{\alpha}} -\frac{1}{2 (p+1)^2 p^{\alpha + 1}}  -\frac{1}{2 (p+1)^2 p^{\alpha + 2}}$$
and when $\left( \frac{s}{p} \right) = -1$,
$$\diop_2^r(\ZZ_p) = \frac{1}{2} + \frac{\alpha (p  - 1)^2 + p^2 + 1}{2p^{\alpha + 2}} - \frac{2p - 1}{2(p+1)^2} +\frac{1}{2 (p+1)^2 p}- \frac{ (\alpha + 1)}{2 (p+1)^2p^{\alpha - 2}} 
$$
$$ +\frac{(\alpha -1) }{2 (p+1)^2 p^{\alpha}} -\frac{1}{2 (p+1)^2 p^{\alpha + 1}}  -\frac{1}{2 (p+1)^2 p^{\alpha + 2}}.$$
\end{proof}

\section{$\Diop_m^r(\ZZ_3)$} \label{sec:Z3}
We can exactly compute the Haar measure of $\D(r)$ $m$-tuples over $\ZZ_3$ for a broad class of values $r$ and all $m\ge 2$. More precisely, we will show the following result:

\begin{thm}\label{t:haarz3}
For any $r \in \ZZ_3$ with $\left( \frac{r}{3} \right) = 1$, the Haar measure of $\D(r)$ $m$-tuples in $\ZZ_3^m$ for $m \ge 2$ is
$$\diop_m^{r}(\ZZ_3) = 
 \frac{m^2 + 71m + 36}{36 \cdot 3^m}.
$$
\end{thm}

Recall that $\alpha \in \square(\ZZ_3)$ iff $v_3(\alpha)$ is even and $$\frac{\alpha}{3^{v_3(\alpha)}} \equiv 1 \pmod 3.$$ 
This leads us to the following observation:

\begin{lemma}\label{l:z31}
Fix $r \in \ZZ_3$ so that $(\frac{r}{3}) = 1$. Then any $\D(r)$ triple $(a, b, c)$ over $\ZZ_3$ has at least one of $a, b, c \equiv 0 \pmod 3$.
\end{lemma}
\begin{proof}
If $(\frac{r}{3}) = 1$, then $r \equiv 1 \pmod 3$.
Suppose to the contrary that $a, b, c \not \equiv 0 \pmod 3$. Then either two of $a, b, c$ are $1 \pmod 3$ or $-1 \pmod 3$. Suppose WLOG that $a \equiv b \pmod 3$. Then $$ab \equiv (1)(1) \equiv (-1)(-1) \equiv  1 \pmod 3$$ and thus $$ab + r \equiv ab + 1 \equiv 2 \pmod 3,$$ so $ab + r$ is not a square over $\ZZ_3$. Thus, $(a, b, c)$ cannot be a $\D(r)$ triple, a contradiction.
\end{proof}

\begin{proof}[Proof of Theorem~\ref{t:haarz3}]
Fix $r \in \ZZ_3$ such that $\left( \frac{r}{3} \right) = 1$. Lemma~\ref{l:z31} implies that for a $\D(r)$ $m$-tuple $(a_1, \ldots a_m)$ with $(\frac{r}{3}) = 1$, all but at most two of the $a_i$'s are congruent to $0 \pmod 3$. We consider cases:
\begin{description}
\item[Case 1] $a_1, \ldots , a_m \equiv 0 \pmod 3$. Then for any $i \neq j$, $$a_i a_j + r \equiv r \equiv 1 \pmod 3 \implies a_i a_j + r \in \square(\ZZ_3).$$
Thus all of the above $m$-tuples are in $\Diop_m^r(\ZZ_3)$. Therefore,
$$\mu \left\{ (a_1, \ldots, a_m) \in \Diop_m^r(\ZZ_3): a_1, \ldots , a_m \equiv 0 \pmod 3 \right\} = \frac{1}{3^m}.$$

\item[Case 2] Exactly one of the $a_i$'s is not congruent to $0\pmod{3}$. 
We have that whenever $i \neq j$
$$a_i a_j + r \equiv r \equiv 1 \pmod 3 \implies a_i a_j + r \in \square ( \ZZ_3),$$
since for all such pairs $i, j$ at least one of $a_i, a_j \equiv 0 \pmod 3$. 
Since there are $m$ ways to choose which element of the triple is not $0 \pmod 3$ and two choices of its nonzero residue mod $3$, we have that
$$\mu \left\{ (a_1, \ldots, a_m) \in \Diop_m^r(\ZZ_3): \exists a_i \not \equiv 0 \pmod 3,\, \forall j \neq i,\, a_j \equiv 0 \pmod 3 \right\} = \frac{2m}{3^m}.$$ 

\item[Case 3a] Suppose exactly two of $a_1, \ldots a_m$ are not congruent to $0\pmod{3}$, $a_i, a_j$, and those two are distinct modulo 3 with $a_i \equiv 1 \pmod 3, a_j \equiv 2 \pmod 3$.
Then without loss of generality $a_1 \equiv 1 \pmod 3, a_2 \equiv 2 \pmod 3$. For $k_1 \neq k_2$, $(k_1, k_2) \neq (i, j), (j, i)$, we have that
$$a_{k_1} a_{k_2} + r \equiv r \equiv 1 \pmod 3 \implies a_{k_1} a_{k_2} + r \in \square(\ZZ_3).$$
Now, we compute the measure of pairs $(a_i, a_j)$ in $\ZZ_3^2$ such that $a_i a_j + r \in \square(\ZZ_3)$. For $k = 1, 2, \ldots $, let 
$$A_{2k} = \{ (a_i, a_j) \in \Diop_2^1(\ZZ_3) : v_3(a_i a_j + r) = 2k \}.$$
Following the method of Lemma~\ref{l:pairresidue}, we find that
$$\mu(A_{2k}) = \frac{3^{2k} (3 - 1) \cdot \frac{3-1}{2}}{3^{2(2k+1)}} = \frac{2}{3^{2k+2}}.$$ 
The measure inside $\ZZ_3^2$ of the desired set is then
$$\sum_{k = 1}^{\infty} \mu(A_{2k}) = \frac{2}{9} \sum_{k = 1}^{\infty} \frac{1}{3^{2k}} = \frac{2}{9} \cdot \frac{1}{9} \cdot \frac{1}{1 - \frac19} = \frac{2}{9} \cdot \frac{1}{8} = \frac{1}{36}.$$

Since there are $m(m-1)$ ways to pick the pair of elements that are $1, -1 \pmod 3$ in the $m$-tuple, 
$$\mu \left\{ (a_1, \ldots, a_m) \in \Diop_m^r(\ZZ_3): \exists i \neq j,\, a_i \equiv 1(3), a_j \equiv 2(3),\, \forall k \neq i, j,\, a_k \equiv 0 (3) \right\} = \frac{m(m-1)}{36 \cdot 3^m}.$$ 
\item[Case 3b] If for $(a_1, \ldots a_m)$ there are exactly 2 elements $a_i, a_j \not \equiv 0 \pmod 3$, but $a_i \equiv a_j \pmod 3$, then $$a_i a_j + r \equiv r + 1 \equiv 2 \pmod 3 \not \in \square(\ZZ_3)$$
meaning that $(a_1, \ldots a_m)$ cannot be a $D(r)$-tuple over $\ZZ_3$.
\item[Case 4] If at least $3$ $a_i$ (WLOG $a_1, a_2, a_3$) are not $0 \mod 3$, then via Lemma~\ref{l:z31}, we can find $i \neq j$ with $a_i a_j + 1 \not \in \square(\ZZ_3)$, so 
$(a_1, \ldots a_m)$ is not a Diophantine $m$-tuple.
\end{description}

The above allows us to compute the Haar measure of $\D(r)$ $m$-tuples in $\ZZ_3^m$ for $r$ such that $(\frac{r}{3}) = 1$ by summing the Haar measures over disjoint subsets of $\ZZ_3^m$.
$$\diop_m^r(\ZZ_3) = \frac{1}{3^m} + \frac{2m}{3^m} + \frac{m(m-1)}{36 \cdot 3^m} = \frac{m^2 + 71m + 36}{36 \cdot 3^m}.$$
\end{proof}

\section{$\Diop_3^r(\ZZ_p)$} \label{sec:D3r}
Throughout this section, we take $p$ to be an odd prime. We give asymptotic bounds on $\diop_3^r(\ZZ_p)$, thereby showing that over triples $(a, b, c) \in \ZZ_p^3$ the events that $(a, b), (b, c), (a, c)$ are $\D(r)$ pairs are asymptotically independent.

\begin{defn}
For a prime $p$, $m \ge 2, r \in \ZZ_p$, we let
$$\widetilde{\Diop}_m^r(\ZZ_p) = \{ (a_1, \ldots, a_m) \in \Diop_m^r(\ZZ_p) : v_p(a_i) = 0,\, \forall i, v_p(a_i a_j + r) = 0 ,\, \forall i \neq j \}$$ and
$$\widetilde{\diop}_m^r(\ZZ_p) = \mu(\widetilde{\Diop}_m^r(\ZZ_p)). $$
Then, let $\widetilde{\Diop}_m^r(\FF_p)$ be the image of $\widetilde{\Diop}_m^r(\ZZ_p)$ under the reduction map $\ZZ_p \rightarrow \FF_p$ and define $\widetilde{\diop}_m^r(\FF_p) = \mu(\widetilde{\Diop}_m^r(\FF_p)) $.
\end{defn}

\begin{thm}\label{t:3asympind}
For fixed $r \in \ZZ$ viewed in $\ZZ_p$, we have that as $p \rightarrow \infty$, 
$$\diop_3^r(\ZZ_p) = \frac{1}{8} + O \left( \frac{1}{p} \right).$$
\end{thm}

To prove Theorem~\ref{t:3asympind}, we will first understand the intermediate quantity $\diop_3^r(\FF_p)$.

\begin{thm}\label{t:3fp} 
For fixed $r \in \ZZ$ viewed in $\FF_p$ with $p \nmid r$ (for odd primes $p$), we have that 
$$\diop_3^r(\FF_p) =\frac18 + \frac{1}{p} \cdot \left( \frac{6 + 3 \left( \frac{r}{p} \right)}{8} \right) + \frac{1}{p^2} \left( \frac{21 + 6 \left( \frac{r}{p} \right)}{8} \right)  + \frac{1}{p^3} \left( \frac{24 \left( \frac{-r}{p} \right) -10 - 21 \left( \frac{r}{p} \right)}{8} \right) .$$
\end{thm}

We recall a result about character sums that will be useful in counting Diophantine $D(r)$-tuples over $\FF_p$:
\begin{lemma}[Theorem 5.48~\cite{LID97}]\label{l:conic}
Let $f(x) = a_2 x^2 + a_1 x + a_0 \in \FF_p[x]$ and $a_2 \neq 0$. Let $\Delta = a_1^2 - 4 a_0 a_2$ be the discriminant of $f$. Then,
$$\sum_{c \in \FF_p} \left( \frac{f(c)}{p} \right) = \begin{cases}
- \left( \frac{a_2}{p} \right) & \Delta \neq 0 \\
(p-1) \left( \frac{a_2}{p} \right) & \Delta = 0.
\end{cases}
$$
\end{lemma}

\begin{lemma}\label{l:3fp1}
For fixed $r \in \ZZ$ viewed in $\FF_p$ that satisfies $\left( \frac{r}{p} \right) = 1$, we have that 
$$\# \left\{ (a, b, c)  \in \Diop_3^r(\FF_p) : abc \equiv 0 \pmod p \right\} = \frac{3p^2 - 1}{2}.$$
\end{lemma}
\begin{proof}
Let $r = s^2 \neq 0$. Note that if at least $2$ of $a, b, c$ are divisible by $p$, then 
$$ab + r \equiv bc + r \equiv ac + r \equiv r = s^2 \not \equiv 0 \pmod p,$$
and thus $(a, b, c) \in \Diop_3^r(\FF_p)$. The size of this subset of $\FF_p^3$ is
$3(p-1) + 1 = 3p - 2$.
Now suppose that exactly one of $a, b, c$ is divisible by $p$, which we shall assume without loss of generality to be $a$. Then $$ab + r \equiv ac + r \equiv r \equiv s^2 \not \equiv 0 \pmod p.$$
Consequently, $(a, b, c) \in \Diop_3^r(\FF_p)$ iff $bc + r \in \square(\FF_p)$.
Since $r = s^2$ and $bc \not \equiv 0 \pmod p$, for each of the $p-1$ choices of $b$, there are $(p-1)/2$ choices of $c$ that yield squares, and thus, we have 
$$3 \cdot (p-1) \cdot \frac{p-1}{2} = \frac{3(p-1)(p-1)}{2}$$
triples $(a, b, c) \in \Diop_3^r(\FF_p)$ with exactly one of $a, b, c \equiv 0 \pmod p$. This gives the desired result:
$$\# \left\{ (a, b, c) \in \Diop_3^r(\FF_p) : abc \equiv 0 \pmod p \right\} = \frac{3p^2 - 1}{2}. $$
\end{proof}

\begin{lemma}\label{l:3fp1-2}
For fixed $r \in \ZZ$ viewed in $\FF_p$ that satisfies $\left( \frac{r}{p} \right) = -1$, we have that 
$$\# \left\{ (a, b, c)  \in \Diop_3^r(\FF_p) : abc \equiv 0 \pmod p \right\} = 0.$$
\end{lemma}
\begin{proof}
If $abc \equiv 0 \pmod p$, then without loss of generality $a \equiv 0 \pmod p$. Then $ab + r \equiv r \pmod p$, so $ab + r \not \in \square(\FF_p)$, since $\left( \frac{r}{p} \right) = -1$. Therefore, $(a, b, c) \not \in \Diop_3^r(\FF_p)$, giving the desired result.
\end{proof}

\begin{lemma}\label{l:3fp2}
For fixed $r \in \ZZ$ viewed in $\FF_p$ such that $p\nmid r$, we have that 
\begin{align*}
&\# \left\{ (a, b, c) \in \Diop_3^r(\FF_p) : (ab + r)(bc + r)(ac + r) \equiv 0 \pmod p,\, abc \not \equiv 0 \pmod p \right\} \\
&= \frac{3p^2+3p+4}{4}-\frac{3p+3}{4}\left(\frac{r}{p}\right)+\frac{13}{4}\left(\frac{-r}{p}\right).
\end{align*} 
\end{lemma}
\begin{proof}
 We consider cases:
\begin{enumerate}
\item Suppose $ab + r \equiv ac + r \equiv ab + r \equiv 0 \pmod p$. 
Then, since $a \not \equiv 0 \pmod p$,
$$ab + r \equiv ac + r \pmod p \implies ab \equiv ac \pmod p \implies b \equiv c \pmod p.$$
By symmetry, this yields, $a \equiv b \equiv c \pmod p$, yielding two points in $\Diop_3^r(\FF_p)$ if $\left( \frac{-r}{p} \right) = 1$, else no point. Thus, there are $$1 + \left( \frac{-r}{p} \right)$$ points in $\Diop_3^r(\FF_p)$ with  $ab + r \equiv ac + r \equiv ab + r \equiv 0 \pmod p$.
\item Now without loss of generality suppose $ab +r \equiv bc + r \equiv 0 \pmod p$ but $ac + r \not \equiv 0 \pmod p$. We consider values for $b$ in a triple $(a, b, c)$ and note that $(a, b, c) \in \Diop_3^r(\FF_p)$ exactly if 
$$ac + r = \frac{r^2}{b^2} + r \in \square(\FF_p) \backslash \{0\}.$$
Note that as $b$ ranges over all nonzero ($ab + r \equiv 0 \pmod p$ implies $b \not \equiv 0 \pmod p$) elements of $\FF_p$, $r^2/b^2$ ranges over all nonzero squares in $\FF_p$. Thus let $r/b = X$ and consider the following quadratic function in $\FF_p[X]$:
$$f(X) = X^2 + r.$$
We observe that for each $X \neq 0$ with $\left(\frac{f(X)}{p} \right) = 1$, we have a single choice of $b = r/X$ that gives a point $(a, b, c) \in \Diop_3^r$. We proceed to count the number of such $X$ as follows:
\begin{align*}
\#\left\{X \in \FF_p^{\times} : \left(\frac{f(X)}{p}\right) = 1\right\} &= \frac12 \sum_{X \in \FF_p^{\times}} \left[ 1 + \left( \frac{f(X)}{p} \right) \right] + \frac12 \left( 1 + \left( \frac{-r}{p} \right) \right) \\
&= \frac12 \sum_{X \in \FF_p} \left(1 + \left( \frac{f(X)}{p} \right) \right) - \frac12 \left( 1 + \left( \frac{r}{p} \right) \right) + \frac12 \left( 1 +  \left( \frac{-r}{p} \right) \right) \\
&= \frac{p}{2} + \frac12 \sum_{x \in \FF_p} \left( \frac{f(X)}{p} \right)  - \frac12 \left( \frac{r}{p} \right) + \frac12 \left( \frac{-r}{p} \right) \\
&\overset{(\ast)}=  \frac12 \left( p - 1 - \left( \frac{r}{p} \right) + \left( \frac{-r}{p} \right) \right),
\end{align*}
where to obtain ($\ast$), we apply Lemma~\ref{l:conic}. Thus, we have $\frac12 \left( p - 1 - \left( \frac{r}{p} \right) + \left( \frac{-r}{p} \right) \right)$ such points $(a, b, c) \in \Diop_3^r(\FF_p)$.

\item Finally, suppose that exactly one of $ab + r, bc + r, ac + r \equiv 0 \pmod p$, WLOG $ab + r \equiv 0 \pmod p$, so $a = -r b^{-1}$. Then, we wish to count the number of choices of $b, c \in \FF_p \backslash \{0\}$ such that 
$$bc + r, (ac + r) = r(1 - cb^{-1}) \in \square(\FF_p) \backslash \{0\}.$$
Since $r \not \equiv 0 \pmod p$, it suffices for $bc + r \in \square(\FF_p) \backslash \{0\},\left( \frac{1 - cb^{-1}}{p} \right) = \left( \frac{r}{p} \right)$. 

Note that for any of the $(p-1)/2$ nonzero values $t^2 \in \FF_p$,  $r(1 - cb^{-1})= t^2$ exactly if we take $c = b (1 - r^{-1} t^2)$. For any such fixed $t^2$, $(a, b, c) \in \Diop_3^r(\FF_p)$ with $bc + r\neq 0$ for all choices of $b$ such that
$$b^2(1 - r^{-1} t^2) + r = u^2 \in \FF_p^{\times}.$$
Since $c \neq 0$, $t^2 \neq r$.
Therefore, we can find all such $b$ by applying Lemma~\ref{l:conic} to
$$f(X) = (1 - r^{-1} t^2)X^2 + r,$$ 
finding that the number of such $b$ is given by 
$$\frac{p}{2} + \frac12 \sum_{x \in \FF_p} \left( \frac{f(X)}{p} \right)  - \frac12 \left( \frac{r}{p} \right) + \frac12 \left( \frac{-r}{p} \right) = \frac{p}{2} - \left( \frac{1 - r^{-1}t^2}{p} \right)  - \frac12 \left( \frac{r}{p} \right) + \frac12 \left( \frac{-r}{p} \right). $$

We can sum over the all the possibilities of $t^2 \in \square(\FF_p^{\times})$ to count the number of $(a, b, c) \in \Diop_3^r(\FF_p)$ with $ab + r \equiv 0 \pmod p$ and $(bc + r)(ac + r)abc \not \equiv 0 \pmod p$:

\begin{align*}
&\# \{ (a, b, c) \in \Diop_3^r(\FF_p) : ab + r \equiv 0 \pmod p,\, (bc + r)(ac + r)abc \not \equiv 0 \pmod p\} \\
&= \sum_{t = 1}^{(p-1)/2} \left[ \frac{p}{2} - \left( \frac{1 - r^{-1}t^2}{p} \right)  - \frac12 \left( \frac{r}{p} \right) + \frac12 \left( \frac{-r}{p} \right) \right] \\
&= \frac{p(p-1)}{4} + \frac{p-1}{4} \left(\left( \frac{-r}{p} \right) - \left( \frac{r}{p} \right) \right) - \sum_{t = 1}^{(p-1)/2} \left( \frac{1 - r^{-1}t^2}{p} \right) \\
&= \frac{p(p-1)}{4} + \frac{p-1}{4} \left(\left( \frac{-r}{p} \right) - \left( \frac{r}{p} \right) \right) - \frac12 \sum_{t \in \FF_p^{\times}} \left( \frac{1 - r^{-1}t^2}{p} \right) \\
&= \frac{p(p-1)}{4} + \frac{p-1}{4} \left(\left( \frac{-r}{p} \right) - \left( \frac{r}{p} \right) \right) - \frac12 \sum_{t \in \FF_p} \left( \frac{1 - r^{-1}t^2}{p} \right) + \frac12 \left( \frac{1}{p} \right) \\
&= \frac{p(p-1)}{4} + \frac{p-1}{4} \left(\left( \frac{-r}{p} \right) - \left( \frac{r}{p} \right) \right) - \frac12 \sum_{t \in \FF_p} \left( \frac{1 - r^{-1}t^2}{p} \right) + \frac12 \\
&\overset{(\ast)}= \frac{p(p-1)}{4} + \frac{p-1}{4} \left(\left( \frac{-r}{p} \right) - \left( \frac{r}{p} \right) \right) + \frac12 \left( \frac{- r^{-1}}{p} \right) + \frac12 \\
&= \frac{p(p-1)}{4} + \frac{p-1}{4} \left(\left( \frac{-r}{p} \right) - \left( \frac{r}{p} \right) \right) + \frac12 \left( \frac{- r}{p} \right) + \frac12,
\end{align*}
where as above $(\ast)$ follows by Lemma~\ref{l:conic} applied to $g(t) = 1 - r^{-1}t^2$. 

\end{enumerate}
Using the symmetry between $a, b, c$, we can use the above three cases to obtain the desired result:
\begin{align*}
&\# \left\{ (a, b, c) \in \Diop_3^r(\FF_p) : (ab + r)(bc + r)(ac + r) \equiv 0 \pmod p,\, abc \not \equiv 0 \pmod p \right\} \\
&= 1 + \left( \frac{-r}{p} \right)+ \frac32 \left( p - 1 - \left( \frac{r}{p} \right) + \left( \frac{-r}{p} \right) \right) \\ &\qquad + 3 \left( \frac{p(p-1)}{4} + \frac{p-1}{4} \left(\left( \frac{-r}{p} \right) - \left( \frac{r}{p} \right) \right) + \frac12 \left( \frac{- r}{p} \right) + \frac12 \right) \\ 
&= \frac{3p^2+3p+4}{4}-\frac{3p+3}{4}\left(\frac{r}{p}\right)+\frac{13}{4}\left(\frac{-r}{p}\right).
\end{align*} 

\end{proof}

\begin{lemma}\label{l:3fp3}
For $r \in \ZZ$ with $p\nmid r$ viewed as an element of $\FF_p$, we have that
$$\widetilde{\diop}_3^r(\FF_p) = \frac18 - \frac{1}{p} \cdot \frac{6 + 3 \left( \frac{r}{p} \right)}{8} + \frac{1}{p^2} \cdot \frac{15 + 12 \left( \frac{r}{p} \right)}{8} - \frac{1}{p^3} \cdot \frac{16 + 13 \left( \frac{r}{p} \right) + 2\left( \frac{-r}{p}\right)}{8}.$$
\end{lemma}
\begin{proof}
We first compute the count of triples in $(a, b, c) \in \Diop_3^r(\FF_p)$ such that $abc(ab +r)(bc + r)(ac + r) \neq 0.$
Let $\alpha = ab, \beta = bc$ and let $\gamma = b^{-1}$, so that $\alpha \beta \gamma^2 = ac$. This gives the following:
\footnotesize
\begin{align*}
 \#&\left\{(a, b, c) \in \Diop_3^r(\FF_p) : (ab + r)(bc + r)(ac + r)abc \not \equiv 0 \pmod{p} \right\} \\
&= \frac18 \sum_{c \in\FF_p^\times} \sum_{b \in\FF_p^\times} \sum_{a \in\FF_p^\times} \mathbb{1}_{\{(ab + r)(bc + r)(ac + r) \neq 0\}}\left(1 + \left( \frac{ab+r}{p} \right) \right) \left(1 + \left( \frac{bc+r}{p} \right) \right) \left(1 + \left( \frac{ac+r}{p} \right) \right) \\
&= \frac18 \sum_{\gamma \in\FF_p^\times} \sum_{\beta \in\FF_p^\times} \sum_{\alpha \in\FF_p^\times} \mathbb{1}_{\{(\alpha + r)(\beta + r)(\alpha \beta \gamma^2 + r) \neq 0\}} \left(1 + \left( \frac{\alpha+r}{p} \right) \right) \left(1 + \left( \frac{\beta +r}{p} \right) \right) \left(1 + \left( \frac{ \alpha \beta \gamma^2 +r}{p} \right) \right) \\
&= \frac18 \sum_{\gamma \in\FF_p^\times} \sum_{\beta \in\FF_p\setminus\{r\}} \sum_{\alpha \in\FF_p\setminus\{r\}} \mathbb{1}_{\{\alpha \beta ((\alpha -r)(\beta - r)\gamma^2 + r) \neq 0\}} \left(1 + \left( \frac{\alpha}{p} \right) \right) \left(1 + \left( \frac{\beta}{p} \right) \right) \left(1 + \left( \frac{ (\alpha - r)(\beta - r)\gamma^2 + r}{p} \right) \right) \\
&= \frac18\sum_{\beta \in\FF_p\setminus\{r\}} \sum_{\alpha \in\FF_p\setminus\{r\}} \mathbb{1}_{\{\alpha \beta \neq 0 \}} \left(1 + \left( \frac{\alpha}{p} \right) \right) \left(1 + \left( \frac{\beta}{p} \right) \right)  \sum_{\gamma \in\FF_p^\times} \mathbb{1}_{\{(\alpha - r)(\beta - r)\gamma^2 + r \neq 0 \}} \left(1 + \left( \frac{ (\alpha - r)(\beta - r)\gamma^2 + r}{p} \right) \right).
\end{align*}
\normalsize
Recall that $p\nmid r$ and first consider the two outer sums. As $\alpha, \beta$ range over the summation, they range over all values in $\FF_p$ except $r$, but contribute a nonzero value only when $\left(\frac{\alpha}{p} \right) = \left(\frac{\beta}{p} \right) = 1$.
We first consider the inner summation, recalling that $\alpha, \beta \neq r$: 
\begin{align*}
&\sum_{\gamma \in\FF_p^\times} \mathbb{1}_{\{(\alpha - r)(\beta - r)\gamma^2 + r \neq 0 \}} \left(1 + \left( \frac{ (\alpha - r)(\beta - r)\gamma^2 + r}{p} \right) \right) \\
&= \sum_{\gamma \in \FF_p} \left[1 + \left( \frac{ (\alpha - r)(\beta - r)\gamma^2 + r}{p} \right) \right] - \left(1  +\left( \frac{r}{p} \right) \right) - 1 \\
&= \left(p - 2 - \left( \frac{r}{p} \right) \right) + \sum_{\gamma \in \FF_p}  \left( \frac{ (\alpha - r)(\beta - r)\gamma^2 + r}{p} \right) \\
&\overset{(\ast)}= p - 2 - \left( \frac{r}{p} \right)  - \left(\frac{(\alpha - r)(\beta - r)}{p} \right) \\
\end{align*}
where as earlier, we apply Lemma~\ref{l:conic} in $(\ast)$.
This simplification gives
\begin{align*}
\frac18&\sum_{\beta = 1+r}^{p-1+r} \sum_{\alpha = 1+r}^{p-1+r} \mathbb{1}_{\{\alpha \beta \neq 0 \}} \left(1 + \left( \frac{\alpha}{p} \right) \right) \left(1 + \left( \frac{\beta}{p} \right) \right)  \sum_{\gamma = 1}^{p-1} \mathbb{1}_{\{(\alpha - r)(\beta - r)\gamma^2 + r \neq 0 \}} \left(1 + \left( \frac{ (\alpha - r)(\beta - r)\gamma^2 + r}{p} \right) \right). \\
&= \frac18\sum_{\beta = 1+r}^{p-1+r} \sum_{\alpha = 1+r}^{p-1+r} \mathbb{1}_{\{\alpha \beta \neq 0 \}} \left(1 + \left( \frac{\alpha}{p} \right) \right) \left(1 + \left( \frac{\beta}{p} \right) \right) \left[ p - 2 - \left( \frac{r}{p} \right)  - \left(\frac{(\alpha - r)(\beta - r)}{p} \right) \right] \\
&= \frac18\sum_{\beta \neq 0, r} \sum_{\alpha \neq 0, r} \left(1 + \left( \frac{\alpha}{p} \right) \right) \left(1 + \left( \frac{\beta}{p} \right) \right) \left( p - 2 - \left( \frac{r}{p} \right) \right) \\
&\qquad \qquad - \frac18\sum_{\beta \neq 0, r} \sum_{\alpha \neq 0, r}\left(1 + \left( \frac{\alpha}{p} \right) \right) \left(1 + \left( \frac{\beta}{p} \right) \right) \left(\frac{(\alpha - r)(\beta - r)}{p} \right)  \\
&= \frac18 \left(p - 2 - \left( \frac{r}{p} \right) \right) \left( p - 2 - \left( \frac{r}{p} \right) \right)^2 - \frac18 \left( \sum_{\alpha \neq 0, r} \left(1 + \left( \frac{\alpha}{p} \right) \right) \left(\frac{\alpha - r}{p} \right) \right)^2 \\
&\overset{(\ast)}= \frac18 \left( p - 2 - \left( \frac{r}{p} \right) \right)^3 - \frac18\left( -1 - \left(\frac{-r}{p} \right)  \right)^2 \\
&= \frac18 \left[ (p^3 - 6p^2 + 15p -16) + \left( \frac{r}{p} \right) (-3p^2 + 12p - 13) -2 \left(\frac{-r}{p} \right) \right]
\end{align*}
where we use Lemma~\ref{l:conic} in $(\ast)$.
This implies that 
$$\widetilde{\diop}_3^r(\FF_p) = \frac18 - \frac{1}{p} \cdot \frac{6 + 3 \left( \frac{r}{p} \right)}{8} + \frac{1}{p^2} \cdot \frac{15 + 12 \left( \frac{r}{p} \right)}{8} - \frac{1}{p^3} \cdot \frac{16 + 13 \left( \frac{r}{p} \right) + 2\left( \frac{-r}{p}\right)}{8}.$$
\end{proof}

\begin{proof}[Proof of Theorem~\ref{t:3fp}]
We add the results of Lemmas~\ref{l:3fp1}--\ref{l:3fp3} to give a unified expression for all $p \nmid r$. We thus obtain the desired Haar measure
$$\diop_3^r(\FF_p) = \frac18 + \frac{1}{p} \cdot \left( \frac{6 + 3 \left( \frac{r}{p} \right)}{8} \right) + \frac{1}{p^2} \left( \frac{21 + 6 \left( \frac{r}{p} \right)}{8} \right)  + \frac{1}{p^3} \left( \frac{24 \left( \frac{-r}{p} \right) -10 - 21 \left( \frac{r}{p} \right)}{8} \right) .$$
\end{proof}

\begin{proof}[Proof of Theorem~\ref{t:3asympind}] 
The result follows by noting that 
$$\widetilde{\diop}_3^r(\ZZ_p) = \widetilde{\diop}_3^r(\FF_p)$$
as if $abc(ab + r)(bc + r)(ac + r) \not \equiv 0 \pmod p$, then $ab + r, bc + r, ac + r \in \square(\FF_p)$ if and only if $ab + r, bc + r, ac + r \in \square(\ZZ_p)$.
Further,
$$\mu\{(a, b, c) \in \ZZ_p^3 \mid abc(ab + r)(bc + r)(ac + r) \equiv 0 \pmod p\} \le \frac{6}{p}.$$
Combining this bound with the result of Theorem~\ref{t:3fp} gives the asymptotic result:
$$\left| \diop_3^r(\ZZ_p) - \diop_3^r(\FF_p) \right| \le \frac{6}{p}.$$
This implies that
$$\diop_3^r(\ZZ_p) = \frac18 + O \left( \frac{1}{p} \right),$$ as desired.
\end{proof}

\section{$\Diop_m^r(\ZZ_p)$ for $m \ge 4$} \label{sec:m4}

In this section, we take $p$ to be an odd prime. It appears to be difficult to compute $\dioph_m^r(\ZZ_p)$ or $\dioph_m^r(\FF_p)$ exactly, but it is possible to give asymptotics. In the case of $r=1$, this was done for $\FF_p$ by Dujella and Kazalicki in~\cite{DK16}; they proved that \[\dioph_m^1(\FF_p)=2^{-\binom{m}{2}}+O(p^{-1/2}).\] Their technique is similar to ours in the previous section, involving character sums evaluated at the values of polynomials. However, the best analogue of Lemma~\ref{l:conic} in this case is only an inequality. To do this in the case of $r=1$, Dujella and Kazalicki consider the sum \[\sum_{a_1,\ldots,a_m\in\FF_p}\prod_{1\le i<j\le m} \frac{1}{2}\left(1+\left(\frac{a_ia_j+1}{p}\right)\right).\] They note that the main contribution to this sum is \[\sum_{a_1,\ldots,a_m\in\FF_p}\frac{1}{2^{\binom{m}{2}}},\] obtained by choosing the 1 from each factor of the form $1+\left(\frac{a_ia_j+1}{p}\right)$. The remaining terms can be bounded using the Weil bound on character sums (or, equivalently, the Hasse--Weil bound for the number of $\FF_p$-points on a genus-$g$ curve): for fixed $m$, the remaining terms are $O(p^{-1/2})$ by~\cite[Theorem 5.41]{LID97}.

The arguments in~\cite{DK16} do not make any use of $r=1$; they only rely on $p\nmid r$. Thus we have the following:

\begin{thm} For all fixed $m\ge 3$ and $r\in\ZZ$, we have that \[\dioph_m^r(\FF_p)=\frac{1}{2^{\binom{m}{2}}}+O\left(\frac{1}{\sqrt{p}}\right)\] as $p\to\infty$. \end{thm}

If we instead work with $\ZZ_p$ instead of $\FF_p$, we have \[\widetilde{\dioph}_m^r(\ZZ_p)=\widetilde{\dioph}_m^r(\FF_p)\] as in the previous section, and that \[|\dioph_m^r(\ZZ_p)-\dioph_m^r(\FF_p)|\le\frac{m+\binom{m}{2}}{p}=O\left(\frac{1}{p}\right),\] so that \[\dioph_m^r(\ZZ_p)=\frac{1}{2^{\binom{m}{2}}}+O\left(\frac{1}{\sqrt{p}}\right)\] as well.

We also note that there is a second technique for estimating $\dioph_4^r(\FF_p)$ and $\dioph_4^r(\ZZ_p)$, which we now describe.

\begin{prop}\label{p:conj4}
For any fixed $r \in \ZZ$ viewed in  $\ZZ_p$, we have that as $p \rightarrow \infty$
$$\diop_4^r(\ZZ_p) = \frac{1}{64} + O\left( \frac{1}{\sqrt{p}} \right).$$
\end{prop}

It is helpful to begin by showing the relevant result for tuples over $\FF_p$:

\begin{prop}
For any fixed $r \in \ZZ$ viewed in  $\FF_p$, we have that as $p \rightarrow \infty$
$$\diop_4^r(\FF_p) = \frac{1}{64} + O\left( \frac{1}{\sqrt{p}} \right).$$
\end{prop}

\begin{proof}
Let $(a, b, c, d) \in \FF_p^4$. Note that $(a, b, c, d) \in \Diop_4^r(\FF_p)$ only if $(a, b, c) \in \Diop_3^r(\FF_p)$ and if $ab + r, bc + r, ac + r \in \square(\FF_p)$. 
Assume that $a, b, c$ are all distinct and let
$$S_0 = \{(a, b, c) \in \Diop_3^r(\FF_p),\, a, b, c \text{ distinct}\}.$$
Fix $(a_0, b_0, c_0) \in S_0$ and consider the elliptic curve $E$ defined as 
\begin{equation}\label{e:e}
E: Y^2 = (a_0X + r)(b_0X + r)(c_0X + r).
\end{equation}
The Hasse--Weil bound gives that 
$$p + 1 - 2\sqrt{p} \le |E(\FF_p)| \le p + 1 + 2\sqrt{p}.$$
Excluding the point at infinity, we find that this gives between $\frac{p - 2\sqrt{p}}{2} - 2$ and $\frac{p + 2\sqrt{p}}{2}$ choices of $d \neq 0 \in \FF_p$, such that there is a point $(X, Y) \in E(\FF_p)$ with $X = d$. Of course, not all of these choices $d \in \FF_p$ make $a_0d + r, b_0d +r, c_0d +r \in \square(\FF_p)$; they merely guarantee that the product $(a_0d + r)(b_0d +r)(c_0d + r) \in \square(\FF_p)$.
For our fixed $(a_0, b_0, c_0) \in S_0$, the values $d \in \FF_p$ that make $(a_0, b_0, c_0, d) \in \Diop_3^r(\FF_p)$ correspond exactly to the choices of $X$ for which there is a $Y\in\FF_p$ with $(X,Y)\in 2E(\FF_p)$ by~\cite[Theorem 4.2]{Knapp92}. 
$E(\FF_p)$ has full 2-torsion as seen via the factored form of $E$ in Equation~\eqref{e:e}. Thus $E(\FF_p) \cong \ZZ/t_1 \ZZ \times \ZZ/t_2 \ZZ$, where $t_1, t_2 \equiv 0 \pmod 2$. Thus, $\frac14$ of the points on $E(\FF_p)$ are in $2 E(\FF_p)$ and $(X, Y) \in 2E(\FF_p)$ if and only if $(X, -Y) \in 2 E(\FF_p)$, which since $p > 2$, gives the following bound (noting that the point at infinity is guaranteed to be a multiple of $2$):
\begin{equation}\label{eq:d}
\frac18p - \frac{\sqrt{p}}{4} -1 \le \#\{ d \in \FF_p : (a_0, b_0, c_0, d) \in \Diop_4^r(\FF_p)\} \le \frac18p + \frac{\sqrt{p}}{4}.
\end{equation}
Notice that
$$\diop_4^r(\FF_p) = \sum_{(a_0, b_0, c_0) \in \Diop_3^r(\FF_p)} \frac{\#\{ d \in \FF_p : (a_0, b_0, c_0, d) \in \Diop_4^r(\FF_p)\}}{p^4}.$$
Further, observe that 
$$| \diop_3^r(\FF_p) - \mu(S_0)| = O\left( \frac{1}{p} \right).$$
Consequently,
$$\mu \{ (a, b, c, d) \in \Diop_4^r(\FF_p) : (a, b, c) \not \in S_0 \} = O\left( \frac{1}{p} \right).$$
Using the bound of Equation~\eqref{eq:d} and recalling that via Theorem~\ref{t:3fp}, $\diop_3^r(\FF_p) = \frac{1}{8} + O \left(\frac{1}{p} \right)$ in conjunction with the above, we find that
\begin{align*}
\diop_4^r(\FF_p) &= \diop_3^r(\FF_p) \cdot \left( \frac{1}{8} + O\left( \frac{1}{\sqrt{p}} \right) \right) + O\left( \frac{1}{p} \right) \\
&= \left(\frac{1}{8} + O \left(\frac{1}{p} \right) \right)\left( \frac{1}{8} + O\left( \frac{1}{\sqrt{p}} \right) \right) + O\left( \frac{1}{p} \right) \\
&= \frac{1}{64} + O\left( \frac{1}{\sqrt{p}} \right).
\end{align*}
\end{proof}

\begin{proof}[Proof of Proposition~\ref{p:conj4}]
Let $(a, b, c, d) \in \ZZ_p^4$. Note that $(a, b, c, d) \in \Diop_4^r(\ZZ_p)$ if and only if both $(a, b, c) \in \Diop_3^r(\ZZ_p)$ and $ab + r, bc + r, ac + r \in \square(\ZZ_p)$. %Thus, we restrict our attention to the set 
For every $(a, b, c) \in \Diop_3^r(\ZZ_p)$, $(a, b ,c) \in \Diop_3^r(\FF_p)$ where we view $a, b, c$ under the natural map $\ZZ_p \rightarrow \FF_p$. For $a_1, a_2 \in \ZZ_p$ unless $ a_1 a_2 + r \equiv 0 \pmod p$, then $a_1 a_2 + r \in \square(\ZZ_p)$ if and only if $a_1 a_2 + r \in \square(\FF_p)$. Thus there are at most $3$ choices of $d \pmod p$, such that $(ad + r)(bd + r)(cd + r) \equiv 0 \pmod p$ for $a, b, c, d \in \FF_p$. This computation implies that 
$$| \diop_4^r(\FF_p) - \diop_4^r(\ZZ_p)| = O \left( \frac{1}{p} \right),$$
and consequently as desired,
$$ \diop_4^r(\ZZ_p) = \frac{1}{64} + O\left( \frac{1}{\sqrt{p}} \right).$$
\end{proof}

A similar technique can be used to show that $\dioph_5^r(\FF_p)$ and $\dioph_5^r(\ZZ_p)$ are $\frac{1}{1024}+O\left(\frac{1}{\sqrt{p}}\right)$, since the curve $y^2=(ax+r)(bx+r)(cx+r)(dx+r)$ is again an elliptic curve. There is an analogue of~\cite[Theorem 4.2]{Knapp92} for hyperelliptic curves (see~\cite{Zarhin16}). Let $C$ be the hyperelliptic curve $y^2=(a_1x+r)\cdots(a_{m-1}x+r)$, where $m-1$ is odd and $a_1,\ldots,a_{m-1}\in\FF_p$ are distinct. Let $(x_0,y_0)\in C(\FF_p)$ be a point. Then $a_ix_0+r$ is a perfect square for all $i$ with $1\le i\le m-1$ if and only if the point $[(x_0,y_0)-\infty]\in 2\Jac(C)(\FF_p)$, where $\Jac(C)$ is the Jacobian of $C$. So, in order to estimate $\dioph_m^r(\FF_p)$, it suffices to understand the proportion of points of the form $[(x_0,y_0)-\infty]$ in $\Jac(C)(\FF_p)$ that are multiples of 2 in $\Jac(C)$. While we know that the 2-torsion in $\Jac(C)(\FF_p)$ is isomorphic to $(\ZZ/2\ZZ)^{2g}$, and so that the proportion of points in $\Jac(C)(\FF_p)$ which are multiples of 2 is $\frac{1}{2^{2g}}$, what we need to know is that of the points of the more specific form $[(x_0,y_0)-\infty]\in\Jac(C)(\FF_p)$, the proportion which are multiples of 2 is $\frac{1}{2^{2g}}$, up to a small error term. This does not appear to be so easy, since we cannot guarantee that points of this specific form in the Jacobian are ``representative'' with respect to being multiples of 2.

\section{$\Diop_m^r(\mcO_K)$ for $K/\QQ_p$}\label{s:extension}
Many of the methods highlighted above carry over when we replace $\ZZ_p$ by a finite extension. Let $K/\QQ_p$ be a finite extension with valuation ring $\mcO_K$ and residue field $\kappa = \FF_q$ for $q = p^f$ and $f \ge 1$. 

Applying the methods of Section~\ref{s:pairs} replacing $p$ by $q$ as appropriate and the Legendre symbol with $\chi$, the quadratic multiplicative character over $\FF_q$, yields the following result:

\begin{thm}\label{t:oddpairs2} Suppose $q$ is a power of an odd prime.
With notation as above, let $\alpha = v_{\pi}(r)$ with $r = \pi^{\alpha} s$ where $\pi$ is a uniformizer and $s \in \mcO_K^{\times}$. Then, we have the following:
\footnotesize
$$\diop_2^r(\mcO_K) = 
\begin{cases}
\frac12 + \frac{1}{q(q+1)} & \chi(r) = 1, \\
\frac12 - \frac{1}{q+1} & \chi(r) = -1, \\
\frac12 - \frac{q-1}{2(q+1)^2} -\frac{(\alpha + 2)
}{2 (q+1)^2 q^{\alpha - 1}} -\frac{1}{2 (q+1)^2 q^{\alpha}} + \frac{(\alpha - 1)  }{2 (q+1)^2 q^{\alpha + 1}}  & \alpha \equiv 1 \pmod 2, \\
\frac{1}{2} + \frac{(\alpha + 1)(q-1)^2}{2q^{\alpha + 2}} - \frac{2q - 1}{2(q+1)^2} +\frac{1}{2 (q+1)^2 q}- \frac{ (\alpha + 1)}{2 (q+1)^2q^{\alpha - 2}} \\
\quad +\frac{(\alpha -1) }{2 (q+1)^2 q^{\alpha}} -\frac{1}{2 (q+1)^2 q^{\alpha + 1}}  -\frac{1}{2 (q+1)^2 q^{\alpha + 2}}
& \alpha > 0,\, \alpha \equiv 0 \pmod 2, \chi(s) = -1, \\
\frac{1}{2} + \frac{\alpha (q  - 1)^2 + q^2 + 1}{2q^{\alpha + 2}} - \frac{2q - 1}{2(q+1)^2} +\frac{1}{2 (q+1)^2 q}- \frac{ (\alpha + 1)}{2 (q+1)^2q^{\alpha - 2}} \\
\quad +\frac{(\alpha -1) }{2 (q+1)^2 q^{\alpha}} -\frac{1}{2 (q+1)^2 q^{\alpha + 1}}  -\frac{1}{2 (q+1)^2 q^{\alpha + 2}}
& \alpha > 0,\, \alpha \equiv 0 \pmod 2, \chi(s) = 1.
\end{cases}
$$
\normalsize
where $\chi$ denotes the quadratic multiplicative character on $\FF_q$.
In particular, we have that the measure of the Diophantine pairs over $\mcO_K$ is given by
$$\diop_2^1(\mcO_K) = \frac12 + \frac{1}{q(q+1)}.$$
\end{thm}

In the case that $K/\QQ_3$ is totally ramified, we obtain an analogue of Theorem~\ref{t:haarz3} in Section~\ref{sec:Z3}:

\begin{thm}
If $K/\QQ_3$ is a finite, totally ramified extension, then for any $r \in \mcO_K$ with $\chi(r) = 1$, the Haar measure of $\D(r)$ $m$-tuples in $\mcO_K^m$ for $m \ge 2$ is
$$\diop_m^r(\mcO_K) = \frac{m^2 + 71m + 36}{36 \cdot 3^m}.$$
\end{thm}

In Sections~\ref{sec:D3r} and~\ref{sec:m4}, the computations of $\Diop_3^r(\ZZ_p)$ and $\Diop_m^r(\ZZ_p)$ extend to a bound when we replace $\ZZ_p$ with $\mcO_K$ for an arbitrary finite $K/\QQ_p$. (Note that the result of Lemma~\ref{l:conic} generalizes to one over $\FF_q$, replacing the Legendre symbol with $\chi$ discussed above.)

\begin{thm}
Let $K_1, K_2, \ldots$ be a sequence of characteristic $0$ local fields, and let $p_i$ be the residue characteristic of $K_i$ such that $p_i \rightarrow \infty$ as $i \to \infty$. Then, as $i \to \infty$,
$$\diop_3^r(\mcO_{K_i}) = \frac18 + O \left( \frac{1}{p_i} \right).$$
\end{thm}

\begin{thm}
Let $m\ge 4$ be an integer, let $K_1, K_2, \ldots$ be a sequence of characteristic $0$ local fields, and let $p_i$ be the residue characteristic of $K_i$ such that $p_i \rightarrow \infty$ as $i \to \infty$. Then, as $i \to \infty$,
$$\diop_m^r(\mcO_{K_i}) = \frac{1}{2^{\binom{m}{2}}} + O \left( p_i^{-1/2} \right).$$
\end{thm}

\bibliography{hypocycloid}
\bibliographystyle{alpha}

\end{document}